\documentclass[10pt]{preprint}
\usepackage[utf8]{inputenc}
\usepackage[OT1]{fontenc}
\usepackage{mathtools}
\usepackage{mhequ}
\usepackage{xcolor}
\usepackage{verbatim}
\usepackage{amsthm}
\usepackage{comment}
\usepackage{shuffle}
\usepackage{hyperref}
\usepackage{bookmark}
\usepackage[scr=boondox, cal=euler, bb=lucida]{mathalfa}
\usepackage[osf]{newtxtext}
\usepackage{newtxmath}
\usepackage[a4paper, margin={3.5cm,2.5cm}]{geometry}
\usepackage{cleveref}
\usepackage{mleftright}
\usepackage{microtype} % improves kerning

\newtheorem{thm}{Theorem}[section]

\newtheorem{lem}[thm]{Lemma}
\newtheorem{prop}[thm]{Proposition}
\newtheorem{cor}[thm]{Corollary}
\newtheorem{defn}[thm]{Definition}
\newtheorem{rem}[thm]{Remark}
\numberwithin{equation}{section}

\crefname{thm}{theorem}{theorems}
\crefname{ass}{assumption}{assumptions}
\crefname{ex}{example}{examples}
\crefname{lem}{lemma}{lemmata}
\crefname{prop}{proposition}{propostions}
\crefname{cor}{corollary}{corollaries}
\crefname{defn}{definition}{definitions}
\crefname{rem}{remark}{remarks}

\let\symb\mathbf
\def\1{\symb{1}}

\def\CC{\mathcal{C}}
\def\D{\mathscr{D}}

\def\Cr{\mathscr{C}}     % PF: NEW, unless we want to give up \C for the complex numbers ...

\def\bF{{\symb F}}
\def\bW{{\symb X}}
\def\bX{{\symb X}}

\def\bW{{\symb W}}

\def\R{{\mathbb R}}

\def\OO{{\mathrm O}}

\let\div\undefined
\DeclareMathOperator{\div}{div}

\colorlet{darkblue}{blue!90!black}
\colorlet{darkred}{red!90!black}

\def\id{\mathrm{id}}

\DeclareMathOperator{\Sh}{Sh}

\def\slash{\unskip\kern0.2em/\penalty\exhyphenpenalty\kern0.2em\ignorespaces}
\def\dash{\unskip\kern0.2em--\penalty\exhyphenpenalty\kern0.2em\ignorespaces}
\author{Carlo~Bellingeri$^1$, Ana~Djurdjevac$^1$, Peter~K.~Friz$^{1,2}$, Nikolas~Tapia$^{1,2}$}
\date{\today}
\title{Transport and continuity equations with (very) rough noise}

\institute{$^1$TU Berlin, $^2$WIAS Berlin\\
\email{\{bellinge,djurdjev,friz,tapia\}@math.tu-berlin.de}}

\begin{document}
\maketitle

\begin{abstract}
Existence and uniqueness for rough flows, transport and continuity equations driven by general geometric rough paths are established.\\[.5ex]
\end{abstract}

\tableofcontents

\section{Introduction}
We consider the {\it transport} equation, here posed (w.l.o.g.) as terminal value problem. This is, 
\begin{equation}\label{eqn:std_RHJ}
\begin{cases}
- \partial_t u (t, x)_{} = \displaystyle\sum_{i=1}^d f_i (x) \cdot D_x u (t, x)\dot{W}_t^i \equiv \Gamma u_t (x) \dot{W}_t \quad\text{ in }\quad(0,T)\times\R^{n},\\[2mm]
u=g \quad\text{ on }\quad\{T\}\times\R^{n}.
\end{cases}
\end{equation}
for fixed $T>0$, with vector fields $f = (f_1,\dotsc,f_d)$ driven by a $\CC^1$-driving signal $W=(W^1,\dotsc,W^d)$. The canonical pairing of $Du =D_x u = (\partial_{x^1}u,\dotsc,\partial_{x^n}u)$ with a vector field is indicated by a dot,
and we already used the operator / vector notation $$\Gamma_i = f_i (x) \cdot D_x, \ \ \Gamma = (\Gamma_1,\dotsc, \Gamma_d).$$ %will be convenient.
By the methods of characteristics, the unique (classical) $\CC^{1, 1}$ transport solution 
$u\colon[0, T] \times \R^n \rightarrow \R$, is given explicitly by
\begin{equ} u (s, x) = u (s, x ; W) : = g (X^{s, x}_T)  \;,\label{equ:C1TE}
\end{equ}
provided $g \in \CC^1$ and the vector fields $f_1,
\ldots, f_d$ are nice enough ($\CC^1_b$ will do) to ensure a $\CC^1$ solution flow for the ODE
\[\begin{cases}
    \dot{X}^{s,x}_t = \displaystyle\sum_{i = 1}^d f_i (X^{s,x}_t) \dot{W}^i_t \equiv f (X_t) \dot{W}_t,\\
    X^{s,x}_s = x\,.
\end{cases}\]
In turn, solving this ODE with random initial data induces a natural evolution of measures, given by the {\it continuity} - or {\it forward equation}
\[\begin{cases}
    \partial_t \rho = \displaystyle\sum_{i = 1}^d \div_x (f_i (x) \rho_t)\,\mathrm d W_t^i   \quad\text{ in }\quad(0,T)\times\R^{n},\\[2mm]
    \rho(0)=\mu \quad\text{ on }\quad\{0\}\times\R^{n} \,.
\end{cases}\]
Well-posedness of the ``trinity'' transport/flow/continuity will depend on the regularity of the data. For $W \in \CC^1$ we have an effective vector field
$$
           b (t,x) = \sum_{i = 1}^d f_i (x) \dot{W}^i_t
$$
which is continuous in $t \in [0,T]$ and inherits the regularity of $f$. In particular, $f \in \CC^1$ will be sufficient for a $\CC^{1,1}$-flow. 
In a landmark paper, DiPerna--Lions \cite{DiPerna1989} and then Ambrsosio \cite{Ambrosio_2004}, showed that the transport problem (weak solutions) is well-posed under bounds on $\div b$ (rather than $D_x b$) which in turn leads to a generalized flow. %For instance, for $b \in L^1(0,T; BV(\mathbb{R}^d))$ and $div b \in L^1(0,T; L^\infty(\mathbb{R}^s))$ the ODE for characteristics has a unique  which in turn leads to well-posedness of weak PDE solutions {\color{red} PLEASE CHECK, also $s$ vs $d$}.
Another fundamental direction may be called {\it regularisation by noise}, based on the observation that generically $\dot{X} = f_0(X) + (noise)$ is much better behaved than the noise-free problem, see e.g. \cite{beck2019,Flandoli_2009, fl2016wellposedness, Catellier2016, catellier2016rough, Maurelli_2011}.

\medskip 

Our work is not concerning with DiPerna-Lions type analysis, nor regularisation by noise. In fact, our driving vector fields will be {\em very smooth}, to compensate for the the irregularity of the noise, which we here assumed to be {\em very rough}. (This trade-off is typical in rough paths and regularity structures.) 

Specifically, we continue a programme started independently by Bailleul--Gubinelli \cite{Bailleul2017} (see also \cite{Deya2019}) and Diehl et al. \cite{Diehl2017} and take $W$ as rough path, henceforth called $\bW$.  As in these works, we are interested in an intrinsic notion of solution. (Rough path stability of transport problems was already noted in \cite{CF09}). The contribution of this article is a treatment of rough noise of arbitrarily low regularity. Based on a suitable definition of solution, carefully introduced below, we can show

\begin{thm} \label{thm:intro}
  Assume $\bW$ is a weakly geometric rough path of H\"older regularity with exponent $\gamma \in (0,1]$. Assume $f$ has $2\lfloor\gamma^{-1}\rfloor+1$ bounded derivatives. Then there is a unique spatially regular (resp. measure-valued) solution to the rough transport (resp. continuity) equation with regular terminal data (resp. measure-valued initial data). 
\end{thm} 
This should be compared with \cite{Bailleul2017, Diehl2017}, which both treat the ``level-$2$ case'', with H\"older noise of exponent $\gamma > 1/3$. Treating the general case, i.e. with arbitrarily small H\"older exponent, requires us in particular to fully quantify the interaction of iterated integrals, themselves constrained by shuffle-relations, and the controlled structure of the PDE problem at hand. In fact, the shuffle relations will be seen crucial to preserve the hyperbolic nature of the rough transport equation. This is different for (ordinary) rough differential equations where the shuffle relations can be discarded at the price of working with branched (think: Ito-type) rough paths. For what it's worth, our arguments restricted to the (well-known) level-$2$-case still contain some worthwhile simplifications with regard to the existing literature, e.g. by avoiding the analysis of an adjoint equation \cite{Diehl2017} and showing uniqueness for weak solutions of the continuity equations via a small class of test functions. On our way we also (have to) prove some facts on (controlled) geometric rough paths of independent interest, not (or only in the branched setting \cite{Trees,hairer2015geometric}) available in the literature.

{\bf Relation to existing works: } 
Unlike the case of rough transport equation, when it comes to stochastic constructions it is impossible to mention all related works stretching over more than four decades, from e.g.  Funaki \cite{funaki1979construction}, Ogawa \cite{ogawa1973partial} to recent works such as \cite{olivera2015density} with fractional noise and Russo--Valois integration.

The many benefits of a robust theory of stochastic partial differential equations, by combining a deterministic RPDE theory with Brownian and more general noise, are now well documented and need not be repeated in detail. Let us still recall one example of interest: multidimensional fractional Brownian motion admits a canonical geometric rough path lift (see e.g. \cite{FH2014}) $1/4 < \alpha$ < H, which constitutes an admissible rough noise for our rough transport and continuity equations. %This may be compared with more classically minded approaches such as \cite{olivera2015density}, using the Russo--Vallois that 
Various authors (see for example Unterberger \cite{Unterberger_2013}, Nualart and Tindel \cite{NT11}, etc.) have constructed ``renormalised'' canonical fractional Brownian rough paths for any $H>0$, fully covered by Theorem \ref{thm:intro}.

\subsection*{Notations}
We fix once and for all a time $T>0$.
In what follows we abbreviate estimates of the form $| (a) - (b) | \lesssim |t-s|^\gamma$ by writing $(a) \underset{\gamma}{=} (b)$.
Given $\gamma\in(0,1)$ we denote by $\CC^\gamma$ the classical Hölder space, i.e. consisting of functions $f\colon[0,T]\to\mathbb R$ such that
\[ \sup_{t\neq s}\frac{|f_t-f_s|}{|t-s|^\gamma}<\infty. \]
Throughout the paper we say {\em geometric} rough path, when we really mean {\em weakly geometric} rough path (since we only work with this type of rough path, the difference \cite{FrizVictoirNoteOnGRP} will not matter to us).

\medskip 

{\bf Acknowledgement.} CB has received funding from DFG research unit FOR2402, AD and NT have received  funding from Excellence Cluster MATH+ (AA4 and EF1, respectively). 
PKF has received funding from the European Research Council (ERC) under the European Union's Horizon 2020 research and innovation programme (grant agreement No. 683164).

\section{Rough paths}
\label{sse:rps}
We start by reviewing the definition of geometric rough paths of roughness $\gamma\in(0,1)$ and controlled rough paths.
We will do so in a Hopf-algebraic language following \cite{hairer2015geometric}, but before we will introduce some basic concepts.

A \emph{word of length $p\ge 1$} over the alphabet $\{1,\dotsc,d\}$ is a tuple $w=(i_1,\dotsc,i_p)\in\{1,\dotsc,d\}^p$, and we set $|w|\coloneq p$.
We denote by $\varepsilon$ the \emph{empty word}, which is by convention the unique word with zero length.
Given two non-empty words $v=(i_1,\dotsc,i_p)$ and $w=(i_{p+1},\dotsc,i_{p+q})$, we denote by $vw\coloneq(i_1,\dotsc,i_p,i_{p+1},\dotsc,i_{p+q})$ their \emph{concatenation}.
By definition $\varepsilon w=w\varepsilon=w$.
We observe that in any case $|vw|=|v|+|w|$.
The concatenation product is associative but not commutative.

The symmetric group $\mathbb S_p$ acts on words of length $p$ by permutation of its entries, that is, $\sigma.w\coloneq(i_{\sigma(1)},\dotsc,i_{\sigma(p)})$.
Given two integers $p,q\ge 1$, a \emph{$(p,q)$-shuffle} is a permutation $\sigma\in\mathbb S_{p+q}$ such that
\begin{equation*}
  \sigma(1)<\sigma(2)<\dotsb<\sigma(p)\enspace\text{and}\enspace\sigma(p+1)<\sigma(p+2)<\dotsb<\sigma(p+q).
\end{equation*}
We denote by $\Sh(p,q)$ the set of all $(p,q)$-shuffles.

\subsection{The shuffle algebra}
The shuffle product was introduced by Ree \cite{Ree} to study the combinatorial properties of iterated integrals, following K.-T. Chen's work.
Let $d\ge 1$ be fixed, and consider the tensor algebra $H$ over $\R^d$, which is defined to be the direct sum
\[ H\coloneq\bigoplus_{p=0}^{\infty} (\R^d)^{\otimes p}. \]
A linear basis for $H$ is given by pure tensors $e_{i_1}\otimes\dotsm\otimes e_{i_p}$, $p\ge 1$ where $\{e_1,\dotsc,e_d\}$ is a basis of $\R^d$, and the additional element $\mathbf{1}$ which  generates $\mathbb{R}^{\otimes 0}\coloneq\R\1$.
In order to ease the notation we denote, for a word $w=(i_1,\dotsc,i_p)$, $e_w\coloneq e_{i_1}\otimes e_{i_2}\otimes\dotsm\otimes e_{i_p}$.
By definition, the set $\{e_w:|w|=p\}$ is a linear basis for $(\R^d)^{\otimes p}$ for any $p\ge 0$.

The space $H$ is endowed with a  product $\shuffle\colon H\otimes H\to H$, called the \emph{shuffle product}, defined on pure tensors as
\[ e_{i_1\dotsm i_p}\shuffle e_{i_{p+1}\dotsm i_{p+q}}=\sum_{\sigma\in\Sh(p,q)}e_{\sigma.(i_1,\dotsc,i_{p+q})}. \]
There is also another operation, called the \emph{deconcatenation coproduct} $\Delta\colon H\to H\otimes H$, defined by
\begin{equation}
  \Delta e_w\coloneq\sum_{uv=w}e_u\otimes e_v.
  \label{eqn:coprod}
\end{equation}
The shuffle product and the deconcatenation coproduct satisfy a compatibility relation (which will not play any role in the sequel), turning the tripe $(H,\shuffle,\Delta)$ into a graded connected bialgebra.
This implies the existence of a linear map $S\colon H\to H$, called the \emph{antipode}, turning $(H,\shuffle,\Delta,S)$ into a Hopf algebra.
In our particular setting, $S$ can be explicitly computed on basis elements by $S(e_{i_1\dotsm i_p})=(-1)^pe_{i_p\dotsm i_1}$.

The coproduct endows the dual space $H^*$ with an algebra structure via the \emph{convolution product} given, for $g,h\in H^*$, by
\[\langle g\star h, x\rangle\coloneq\langle g\otimes h,\Delta x\rangle. \]
On pure tensor this yields
\[ \langle g\star h,e_w\rangle=\sum_{uv=w}\langle g,e_u\rangle\langle h,e_v\rangle. \]
A \emph{character} is a linear map $g\in H^*$ such that $\langle g,x\shuffle y\rangle=\langle g,x\rangle\langle g,y\rangle$ for all $x,y\in H$.
It is a standard result (see e.g. \cite{Man2008}) that the collection of all characters on $H$ forms a group $G$ under the convolution product whose identity is the function $\mathbf{1}^*\in H^*$, defined  by $\mathbf{1}^*(e_b)=0$ for every word $b$ and $\mathbf{1}^*(\mathbf{1})=1$.
The inverse of an element $g\in G$ can be computed by using the antipode: $g^{-1}=g\circ S$.

Given $N\ge 0$, we consider the \emph{step-$N$ truncated tensor algebra}
\[ H_N=\bigoplus_{p=0}^N(\R^d)^{\otimes p}. \]
\begin{defn}A \emph{step-$N$ truncated character} is a linear map $g\in H_N^*$ such that
  \begin{equation}\label{equ:character}
  \langle g,x\shuffle y\rangle=\langle g,x\rangle\langle g,y\rangle
\end{equation}
for all $x\in(\R^d)^{\otimes p}$ and $y\in(\R^d)^{\otimes q}$ with $p+q\le N$.
\end{defn}
It is not hard to show that the set $G^{(N)}$ of all step-$N$ truncated characters is also a group under the convolution product, whose identity is again $\mathbf{1}^*$.
Denoting by $e_1^*,\dotsc,e^*_d$ the basis of $\R^d$ dual to $\{e_1,\dotsc,e_d\}$, we introduce the dual basis $(e_a^*)$ of $H_N^*$ in the canonical way, that is, for a word $w$ we denote by $e_w^*$ the unique linear map on $H_N$ such that
\[ \langle e_w^*,e_v\rangle=\delta_w(v). \]
The convolution product of two of these basis elements can be explicitly computed.
Indeed, by definition
\[ \langle e_u^*\star e_v^*,e_w\rangle=\sum_{u'v'=w}\langle e_u^*,e_{u'}\rangle\langle e_v^*,e_{v'}\rangle \]
which is nonzero if and only if $w=uv$, in which case $\langle e_u^*\star e_v^*,e_w\rangle=1$.
Therefore $e_u^*\star e_v^*=e_{uv}^*$.
For this reason this product is also known as the \emph{concatenation product}.

\subsection{Geometric rough paths}
We now recall the notion of geometric rough paths.
The group $G^{(N)}$ can be endowed with a \emph{sub-additive homogeneous norm} $\|\cdot\|_N\colon G^{(N)}\to\R_+$, see \cite{LV07extension} for further details.
This allows us to define a left invariant metric on $G^{(N)}$ by setting
\[ d_N(g,h)\coloneq\|h^{-1}g\|_N. \]

\begin{defn}
  Let ${N_{\gamma}}\coloneq \lfloor\gamma^{-1}\rfloor$ denote the integer part of $\gamma^{-1}$. A \emph{geometric rough path of regularity $\gamma$} is a $\gamma$-Hölder path $\bW\colon[0,T]\to(G^{(N_\gamma)},d_N)$.
The set of all geometric rough paths of regularity $\gamma$ will be denoted by $\Cr^\gamma$.
\label{defn:grp}
\end{defn}

By definition of the increments $\bW_{st}\coloneq \bW_s^{-1}\star\bW_t$ satisfy the so-called Chen's relations
\begin{equation}\label{equ:chen}
\bW_{st}=\bW_{su}\star\bW_{ut}
\end{equation}
for all $0\le s,u,t\le T$.
Moreover, by construction of the homogeneous norm $\Vert \cdot \Vert_{N}$, for any word $w$  such that $|w|\leq N_{\gamma}$ one has
\begin{equation}\label{equ:Holder}
\sup_{t\neq s}\frac{|\langle\bW_{st},e_w \rangle|}{|t-s|^{|w|\gamma}}<\infty.
\end{equation}

\subsection{Controlled rough paths and rough integrals}
One of the main goals of rough paths theory is to give meaning to solutions of controlled equations of the form
\begin{equation}\label{eqn:RDE}
  \mathrm dX_t=\sum_{i=1}^df_i(X_t)\,\mathrm d\bW^i_t,
\end{equation}
for some collection of sufficiently regular vector fields $f_1,\dotsc, f_d$ on $\R^n$ and where the driving signals $W^1,\dotsc, W^d$ are very irregular. The general philosophy is that if the smoothness of the vector fields compensates the lack of regularity of the driving signals, then we can still have existence of solutions given that we reinterpret the equation in the appropriate sense. The central ingredient for proving this kind of results is the notion of \emph{controlled rough path} which we now recall.

\begin{defn}[\cite{Trees,FH2014}]\label{defn:cRP}
  Let $\bW\in \Cr^{\gamma}$ and $1\leq N\leq N_{\gamma}+1$.  A \emph{rough path controlled by $\bW$} is a path $\bX\colon[0,T]\to H_{N-1}$ if for any word $w$ such that $\vert w\vert \leq N-1$ the path $t\mapsto\langle e_w^*,\bX_t\rangle\in\CC^\gamma$ and
\begin{equation} \label{eqn:ctrlbound}
\langle e_w^*,\bX_t\rangle\underset{(N-|w|)\gamma}{=}\langle\bW_{st}\star e_w^*,\bX_s\rangle\,.
\end{equation}
for all $s<t$. We denote by $\D^{N\gamma}_{\bW}$ the (vector) space of paths $\bX$ satisfying \eqref{eqn:ctrlbound}.

We say that a path $X\colon[0,T]\to\R$ is \emph{controlled by $\bW$} if there exists a controlled path $\bX\in\D^{N\gamma}_{\bW}$ such that $\langle\1^*,\bX_t\rangle=X_t$; we call $\bX$ a controlled rough path above (the controlled path) $X$.
\end{defn}

\begin{rem}
  The definition in \cite{FH2014} seems more restrictive in that one always take $N=N_\gamma$, which is the minimal value of $N$ required for rough integration. The case $N=N_\gamma+1$ is convenient to keep track of the additional information obtained by rough integration, see \Cref{rem:rintbound}.
\end{rem}

\begin{rem}
Alternatively, by writing $\bX$ and $\bW$ as the sums
\[\bX_s=\sum_{|w|\le N-1}\langle e_w^*,\bX_s\rangle e_w,\quad \bW_{st}=\sum_{|v|\le N}\langle \bW_{st}, e_v\rangle e_v^*,\]
the condition in \cref{eqn:ctrlbound} can be explicitly written
\begin{equation}\label{eqn:ctrlcoord}
\langle e_w^*,\bX_t\rangle \underset{(N-|w|)\gamma}{=}\sum_{0\leq |v|\le N-|w|}\langle e_{wv}^*,\bX_s\rangle\langle\bW_{st},e_v\rangle,
\end{equation}
for any word $w$.
\label{rem:coordinates}
\end{rem}
By construction of the vector space $\D^{N\gamma}_{\bW}$, the quantity
\[\lVert\mathbf{X} \rVert_{\bW;N\gamma}:= \sum_{0\leq \vert w\vert <N} \sup_{s<t}\frac{|\langle e_w^*,\bX_t\rangle-\langle\bW_{st}\star e_w^*,\bX_s\rangle|}{|t-s|^{(N- \vert w \vert)\gamma}}\,,\]
is finite for any $\bX\in \D^{N\gamma}_{\bW}$. We can easily show that $ \Vert\cdot \Vert_{\D^{N\gamma}_{\bW}} $ is a seminorm and, it becomes a Banach space under the norm
\[ \lVert\bX\rVert_{\D^{N\gamma}_{\bW}}\coloneq\max_{|w|\le N-1}|\langle e_w^*,\bX_0\rangle|+\Vert\bX \Vert_{\bW;N\gamma}.  \]

We extend the notion of controlled rough path above a vector-valued path $X\colon [0,T]\to \mathbb{R}^n$. In this case, the path $\bX$ takes values in $(H_{N-1})^n$, that is, each component path $\langle e_w^*,\bX\rangle$ is a vector of $\R^n$, which we denote by
\[ \langle e_w^*,\bX\rangle=(\langle e_w^*,\bX\rangle_1,\dotsc,\langle e_w^*,\bX\rangle_n). \]
Then we require the bound in \cref{eqn:ctrlbound} to hold componentwise, or equivalently, we can replace the absolute value of the left-hand side by any norm on $\R^n$. We denote this space by $(\D^{N\gamma}_{\bW})^{n}$.

Using the higher-order information contained in the controlled rough path $\mathbf{X}\in \D^{N\gamma}_{\bW}$, we recall the rigorous notion of rough integral of $\bX$ against $\bW$. For its proof see \cite{FH2014}.
\begin{thm}
  Let $\bW\in \Cr^{\gamma}$ and $\bX\in \D^{N_{\gamma}\gamma}_{\bW}$. For every $i\in \{1,\dotsc,d\}$ there exists a unique real valued path in $\CC^{\gamma}$

\begin{equation}\label{eqn:rintdef}
t\mapsto\int_0^t X_u\,\mathrm d\bW^i_u\coloneq\lim_{|\pi|\to 0}\sum_{[a,b]\in\pi}\sum_{0\leq |w|\le N_{\gamma}-1}\langle e^*_w,\bX_a\rangle \langle\bW_{ab}, e_{wi}\rangle,
\end{equation}
where $\pi$ is a sequence of partitions of $[0,t]$ whose mesh $|\pi|$ converges to $0$. We call it\emph{ the rough integral of $X$ with respect to $W^i$}. Moreover one has the estimate
\begin{equation}\label{equ:rough_int_est}
  \int_0^tX_u\,\mathrm d\bW^i_u- \int_0^sX_u\,\mathrm d\bW^i_u\eqcolon\int_s^tX_u\,\mathrm d\bW^i_u\underset{(N_{\gamma}+1) \gamma}{=}\sum_{0<|w|\leq  N_{\gamma}}\langle e^*_w,\bX_s\rangle \langle\bW_{st}, e_{wi}\rangle,
\end{equation}
for any $s<t$. Introducing the function $\int_0^{\cdot}\bX_u\,\mathrm d\bW^i_u\colon [0,T]\to H_{N_{\gamma}}$ given by
\begin{equation}\label{eqn:roughint_coef}
\left\langle\1^*,\int_0^t\bX_u\,\mathrm d\bW^i_u\right\rangle\coloneq\int_0^tX_u\,\mathrm d\bW^i_u\,,\quad\left\langle e_{wi}^*,\int_0^t\bX_u\,\mathrm d\bW^i_u\right\rangle\coloneq\langle e_w^*,\bX_t\rangle\,
\end{equation}
and zero elsewhere, one has  $\int_0^{\cdot}\bX_u\,\mathrm d\bW^i_u\in \D^{(N_{\gamma}+1)\gamma}_{\bW}$.
\end{thm}
\begin{rem}
  \label{rem:rintbound}
  Differently from the general definition of the $\D^{N\gamma}_{\bW}$ spaces, in order to define the rough integral it is necessary to start from a controlled rough path $\bX\in \D^{N_{\gamma}\gamma}_{\bW}$. The operation of integration on controlled rough path comes also with some quantitative bounds. Looking at the definition, it is also possible to prove there exists a constant $C(T, \gamma, \bW)>0$ depending on $T$, $\gamma$, $\bW$ such that
  \[\left\Vert \int_0^{\cdot}\bX_u\,\mathrm d\bW^i_u  \right\Vert_{\D^{(N_\gamma+1)\gamma}_{\bW}}\leq C(T, \gamma, \bW)\lVert\bX\rVert_{\D^{N_\gamma\gamma}_{\bW}}.\]
  Therefore the application $\bX\mapsto \int \bX\,\mathrm d\bW^i$ is a continuous linear map.
\end{rem}

The second operation we introduce is the composition of a controlled rough path and a smooth function. Given a smooth function $\phi\colon\R^n\to\R$, its $k$-th derivative at $x\in\R^n$ is the multilinear map $D^k\phi(x)\colon(\R^n)^{\otimes k}\to\R$ such that for $v^1,\dotsc,v^k\in\R^n$,
\begin{equation}
  D^k\phi(x)(v^1,\dotsc,v^k)=\sum_{\alpha_1,\dotsc,\alpha_k=1}^n\frac{\partial^k\phi}{\partial x_{\alpha_1}\dotsm\partial x_{\alpha_k}}(x)v_{\alpha_1}^1\dotsm v_{\alpha_k}^k.
  \label{eqn:diffk}
\end{equation}

To ease notation we define
\[ \partial^\alpha \phi(x)\coloneq\frac{\partial^k\phi}{\partial x_{\alpha_1}\dotsm\partial x_{\alpha_k}}(x) = \frac{\partial^k\phi}{\partial x^{i_1}_{1}\dotsm\partial x^{i_n}_{n}}(x) \]
for a word $\alpha=(\alpha_1,\dotsc,\alpha_k)\in\{1,\dotsc,n\}^k$; of course, such $\alpha$ induces a multi-index $i=(i_1,\dotsc,i_n)\in\mathbb N^n$, where $i_j$ counts the number of entries of $\alpha$ that equal $j$.

We note that $D^k\phi(x)$ is symmetric, meaning that for any permutation $\sigma\in\mathbb S_k$ we have that $D^k\phi(x)(v^1,\dotsc,v^k)=D^k\phi(x)(v^{\sigma(1)},\dotsc,v^{\sigma(k)}).$

\begin{rem}
  Observe that we also use the notion of word in this case, albeit with a different alphabet.
  In order to avoid confusion we reserve latin letters such as $u,v,w$, etc for words on the alphabet $\{1,\dotsc,d\}$, introduced in the beginning of \Cref{sse:rps}, and greek letters such as $\alpha,\beta$, etc for words on the alphabet $\{1,\dotsc,n\}$ as above.
\end{rem}

With these notations, Taylor's theorem states that if $\phi\colon\R^n\to\R^m$ is of class $\CC^{r+1}(\R^n, \R^m)$ then for any $j=1,\dotsc,m$ one has the identity
\begin{equation}\label{Taylor_thm}
  \phi^j(y)=\sum_{k=0}^{r}\frac{1}{k!}D^k\phi^j(x)\left( (y-x)^{\otimes k} \right)+\OO(|y-x|^{r+1})
 \end{equation}

In what follows, for any finite number of words $u_1,\dotsc,u_k$ we introduce the set
\[ \Sh(u_1,\dotsc,u_k)\coloneq\{w:\langle e_w^*,e_{u_1}\shuffle \dotsc \shuffle e_{u_k}\rangle\neq 0\}. \]
Since the shuffle product is commutative, for any permutation $\sigma\in\mathbb S_k$ we have that
\[ \Sh(u_1,\dotsc,u_k)=\Sh(u_{\sigma(1)},\dotsc,u_{\sigma(k)}). \]

Thanks to this notation, we can prove Faà di Bruno's formula (see also \cite{MR2200529}).
We denote by $\mathcal P(m)$ the collection of all partitions of $\{1,\dotsc,m\}$.
Given $\pi=\{B_1,\dotsc,B_p\}\in\mathcal P(m)$, we let $\#\pi\coloneq p$ denote the number of its blocks, and for each block we denote by $|B|$ its cardinality.

\begin{lem}
  \label{lem:faadibruno}
  For any couple of functions $g\colon \R^n\to \R^n$ and $f\colon \R^n\to \R$ sufficiently smooth and every $m\ge 1$, letting $h\coloneq f\circ g$ one has the identity
  \[ D^mh(x)(v_1,\dotsc,v_m)=\sum_{\pi\in\mathcal P(m)}D^{\#\pi}f(g(x))(D^{|B_1|}g(x)(v_{B_1}),\dotsc,D^{|B_p|}g(x)(v_{B_p})) \]
  where $v_B\coloneq(v_{i_1},\dotsc,v_{i_q})$ for $B=\{i_1,\dotsc,i_q\}$.

  In particular, for any word $\alpha=(\alpha_1,\dotsc,\alpha_m)$ we have
  \begin{equation}\label{Faa_di_bruno}
    \partial^{\alpha}h(x)=\sum_{k=1}^{m}\frac{1}{k!}\sum_{\substack{\beta_1,\dotsc,\beta_k \\ \alpha\in\operatorname{Sh}(\beta_1,\dotsc,\beta_k)}}D^kf(g(x))(\partial^{\beta_1}g(x),\dotsc,\partial^{\beta_k}g(x)).
  \end{equation}
\end{lem}
\begin{proof}
  We proceed by induction on $m$. For $m=1$ the formula reads
  \[ Dh(x)v=Df(g(x))Dg(x)v \]
  which is the usual chain rule.
  Suppose the formula holds for some $m\ge 1$. Then, applying the chain rule to each of the terms we get
  \begin{align*}
    D^{m+1}h(x)(v_1,\dotsc,v_{m+1})&=\sum_{\pi\in\mathcal P(m)}\sum_{l=1}^{k}D^{\#\pi+1}f(g(x))\left( D^{|B_1|}g(x)v_{B_1},\dotsc,D^{|B_l|+1}g(x)(v_{B_l},v_{m+1}),\dotsc,D^{|B_k|}g(x)v_{B_k} \right)\\
      &\qquad+\sum_{\pi\in\mathcal P(m)}D^{\#\pi+1}f(g(x))(D^{|B_1|}g(x)v_{B_1},\dotsc,D^{|B_k|}g(x)v_{B_k},Dg(x)v_{m+1})\\
      &= \sum_{\pi'\in\mathcal P(m+1)}D^{\#\pi'}f(g(x))\left( D^{|B'_1|}g(x)(v_{B'_1}),\dotsc,D^{|B'_p|}g(x)(v_{B'_{k'}}) \right)
  \end{align*}
  where the last identity follows from the fact that every partition $\pi'\in\mathcal P(m+1)$ can be obtained by either appending $m+1$ to one of the blocks of some partition $\pi\in\mathcal P(m)$ or by adding the singleton block $\{m+1\}$ to it.

  Given a word $\alpha=(\alpha_1,\dotsc,\alpha_m)$, we evaluate the previous formula in the canonical basis vectors $v_1=e_{\alpha_1},\dotsc,v_m=e_{\alpha_m}$ to obtain
  \begin{align*}
    \partial^\alpha h(x)&= D^mh(x)(v_1,\dotsc,v_m)\\
    &= \sum_{\pi\in\mathcal P(m)}D^{\#\pi}f(g(x))\left( \partial^{\alpha_{B_1}}g(x),\dotsc,\partial^{\alpha_{B_k}}g(x) \right)
  \end{align*}
  where $\alpha_{B}=(\alpha_{i_1},\dotsc,\alpha_{i_q})$ if $B=\{i_1,\dotsc,i_p\}$.
  It is now clear that for any choice of $\pi\in\mathcal P(m)$ the words $\alpha_{B_1},\dotsc,\alpha_{B_k}$ satisfy $\alpha\in\Sh(\alpha_{B_1},\dotsc,\alpha_{B_k})$.
  Conversely, if $\alpha\in\Sh(\beta_1,\dotsc,\beta_k)$, there is a partition $\pi=\{B_1,\dotsc,B_k\}$ with $B_j=$ is such that $\beta_j=\alpha_{B_j}$.
  Moreover, for any choice of such a partition, any of the $k!$ permutations of its blocks result in the same evaluation by symmetry of the differential.
  Thus
  \[ \partial^\alpha h(x)=\sum_{k=1}^m\frac{1}{k!}\sum_{\substack{\beta_1,\dotsc,\beta_k\\ \alpha\in\Sh(\beta_1,\dotsc,\beta_k)}}D^kf(g(x))(\partial^{\beta_1}g(x),\dotsc,\partial^{\beta_k}g(x)). \]
\end{proof}

\begin{rem}
  This result should be well-known to experts, yet the closest reference we found in the literature [Har06] only covers the scalar case (and does not immediately yield the multivariate case).
\end{rem}

Using a similar technique we show a version of this identity for controlled rough paths.
\begin{thm}\label{thm:ito}
  Let $\bW\in\Cr^\gamma$, $1\le N\le N_\gamma+1$, $\bX\in(\D_{\bW}^{N\gamma})^n$, and $\phi\in C^{N}(\R^n, \R^m)$ and set $X_t\coloneq\langle \mathbf{1},\bX_t\rangle$. We introduce the path $\mathbf{\Phi}(\bX)\colon[0,T]\to(H_{N-1})^m$ defined by $\langle \mathbf{1}^*,\mathbf{\Phi}(\bX)_t\rangle_j= \phi^j(X_t)$ and for any $j=1,\dotsc,m$, and any non-empty word $w$ by the identity
 \begin{equation} \label{eqn:taylorctrl}
   \langle e_w^*,\mathbf{\Phi}(\bX)_t\rangle_j\coloneq\sum_{k=1}^{|w|}\frac1{k!}\smashoperator[r]{\sum_{\substack{u_1,\dotsc,u_k\\w\in\Sh(u_1,\dotsc,u_k)}}}D^k\phi^j(X_t)(\langle e_{u_1}^*,\bX_t\rangle,\dotsc,\langle e_{u_k}^*,\bX_t\rangle).
 \end{equation}
Then $\mathbf{\Phi}(\bX)$ is also a controlled rough path belonging to $(\D_{\bW}^{N\gamma})^m$.
\end{thm}

\begin{rem}
  A similar statement in the setting of branched rough paths \cite[Lemma 8.4]{Trees} is known and somewhat easier due to the absence of shuffle relations. 
\end{rem}
Before going into the proof, we introduce some more notation.
If $\bX$ is a controlled path, $L\in\mathcal L( (\R^n)^{\otimes k},\R^m)$, $t\ge 0$ and $u_1,\dotsc,u_k$ are words, we let
\begin{equation*}
  L(t;u_1,\dotsc,u_k)\coloneq L(\langle e_{u_1}^*,\bX_t\rangle,\dotsc,\langle e_{u_k}^*,\bX_t\rangle)\\
\end{equation*}
\begin{proof}
It is sufficient to prove the result when $m=1$. We first prove the result for the case of $\langle\mathbf{1}^*,\mathbf{\Phi}(\bX)_t\rangle=\phi(X_t)$. By Taylor expanding $\phi$ up to order $N$ around $X_s$ we get that
\begin{equation*}
  \phi(X_t)\underset{N\gamma}{=}\sum_{k=0}^{N-1}\frac{1}{k!}D^k\phi(X_s)\left( (X_t-X_s)^{\otimes k} \right)
\end{equation*}
Since $\bX\in\left( \D_{\bW}^{N\gamma} \right)^n$, according to \Cref{rem:coordinates}, we have
\begin{equation}\label{estimate_X}
\langle\1^*,\bX_t-\bX_s\rangle\underset{N\gamma}{=}\langle\bW_{st}-\1^*,\bX_s\rangle=\sum_{0<|u|< N}\langle e_u^*,\bX_s\rangle\langle\bW_{st},e_u\rangle.
\end{equation}
Plugging this estimate into the above equation and using the character property of $\bW_{st}$ in \eqref{equ:character} we obtain
\begin{align*}
  \phi(X_t)&\underset{N\gamma}{=}\sum_{k=0}^{N-1}\frac{1}{k!}\sum_{u_1,\dotsc,u_k}D^k\phi(X_s)(s;u_1,\dotsc,u_k)\langle\bW_{st},e_{u_1}\shuffle\dotsm\shuffle e_{u_k}\rangle\\
  &= \sum_{k=0}^{N-1}\frac{1}{k!}\sum_{u_1,\dotsc,u_k}\sum_{|w|\le N}D^k\phi(X_s)(s;u_1,\dotsc,u_k)\langle e_w^*,e_{u_1}\shuffle\dotsm\shuffle e_{u_k}\rangle\langle\bW_{st},e_w\rangle
\end{align*}
so the desired estimate follows.

Now we show the bound \eqref{eqn:ctrlbound} for all words $w\neq \mathbf{1}$.
By fixing an integer $1\leq k\leq \vert w \vert$ and words $u_1,\dotsc,u_k$  such that $w\in\Sh(u_1,\dotsc,u_k)$ we consider the term
\begin{equation}\label{equ:higher_terms}
  D^k\phi(X_t)(t;u_1,\dotsc,u_k).
\end{equation}
Again, since $\bX$ is controlled by $\bW$, plugging the estimate in \Cref{rem:coordinates} into \eqref{equ:higher_terms} and using the multilinearity :of the derivative we obtain
\begin{equation}\label{second_estimate}
    D^k\phi(X_t)(t;u_1,\dotsc,u_k)\underset{(N -\vert w\vert)\gamma}{=}\sum_{v_1,\dotsc,v_k}D^k\phi(X_t)(s;u_1v_1,\dotsc,u_kv_k)\langle\bW_{st},e_{v_1}\shuffle\dotsm\shuffle e_{v_k}\rangle.
\end{equation}
Performing a Taylor expansion of $D^k\phi$ up to order $N-|w|$ between $X_t$ and $X_s$, we obtain 
\begin{equation}\label{third_estimate}
  D^k\phi(X_t)(s;u_1v_1,\dotsc,u_kv_k)\underset{(N -\vert w\vert)\gamma}{=}\sum_{m=0}^{N-\vert w\vert-1}\frac{1}{m!}D^{k+m}\phi(X_s)\left((X_t-X_s)^{\otimes m} , \langle e_{u_1v_1}^*,\bX_s\rangle,\dotsc,\langle e_{u_kv_k}^*,\bX_s\rangle\right).
\end{equation}
Combining the estimates \eqref{second_estimate} and \eqref{third_estimate} with \eqref{estimate_X} into the definition of $\langle e_w^*,\mathbf{\Phi}(\bX)_t\rangle$, we obtain the identity
\begin{equation}\label{hard_combinatorics}
\begin{split}\langle e_w^*,\mathbf{\Phi}(\bX)_t\rangle&\underset{(N -\vert w\vert)\gamma}{=}
  \sum_{k=1}^{|w|}\sum_{m=0}^{N-1- \vert w\vert}\sum_{\substack{u_1,\dotsc,u_k\\w\in\Sh(u_1,\dotsc,u_k)}}\sum_{\substack{v_1,\dotsc,v_k\\z_1,\dotsc,z_m}}\frac{1}{k!m!}D^{k+m}\phi(X_s)(u_1v_1,\dotsc,u_kv_k,z_1,\dotsc,z_m)\\&\hspace{10em}\times\langle\bW_{st},e_{v_1}\shuffle\dotsm\shuffle e_{v_k}\shuffle e_{z_1}\shuffle\dotsm\shuffle e_{z_m} \rangle.
\end{split}
\end{equation}
Since the derivative $D^{k+m}\phi(X_s)$ is symmetric we can replace it with
\[\frac{k!m!}{(k+m)!}\sum_{ I_k\sqcup J_m= \{1,\dotsc,m+k\}} D^{k+m}\phi(X_s)(u_{i_1}v_{i_1},\dotsc,z_{j_1},\dotsc,u_{i_k}v_{i_k},\dotsc).\]
Replacing this expression in the right-hand side of \eqref{hard_combinatorics}, it is now an  easy but tedious exercise to  verify the resulting expression is equal to the sum
\[ \sum_{0\leq  \vert u \vert< N-\vert w\vert }\sum_{l=1}^{|w|+ \vert u\vert}\sum_{\substack{v_1,\dotsc,v_l \\ wu\in\operatorname{Sh}(v_1,\dotsc,v_l)}}\frac{1}{l!}D^l\phi(X_s)(s;v_1,\dots,v_l)\langle\bW_{st},e_{c}\rangle.\]
Thereby proving the result.
\end{proof}
\begin{rem}
  A similar proof gives quantitative bounds on the application $\bX\to \mathbf{\Phi}(\bX)$. Indeed for any $\phi\in \CC^{N}_b(\R^n, \R^m)$ it is possible to prove that this application is locally Lipschitz on $\D^{N\gamma}_{\bW}$.
\end{rem}

\section{Rough Differential Equations}
\label{sse:rdes}
Now we come to the definition of solution of the RDE
\begin{equation}\label{equ:RDE}
\begin{cases}
  \mathrm dX_t=\displaystyle\sum_{i=1}^df_i(X_t)\,\mathrm d\bW^i_t,\\
X_0=x.
\end{cases}
\end{equation}
We assume that the vector fields $f_1,\dotsc,f_d$ are of class at least $\CC^{N_\gamma}$, so that by \Cref{thm:ito} the composition $f_i(X_t)$ can be lifted to a controlled path $\bF_i\colon \left( \D^{N_{\gamma}\gamma}_{\bW} \right)^n\to \left( \D^{N_{\gamma}\gamma}_{\bW} \right)^n$.
\begin{defn} \label{defn:rdesol}
  A path $X\colon[0,T]\to \R^n$ is a solution of \eqref{equ:RDE} if there exists a controlled path $\bX\in\left( \D^{N_{\gamma}\gamma}_{\bW} \right)^n$ satisfying $\langle \mathbf{1}^*, \bX_t\rangle= X_t$ such that
\begin{equation}\label{equ:forw_fixed_point}
\bX_t-\bX_s=\sum_{i=1}^d\int_s^t\bF_i(\bX)_u\,\mathrm d\bW^i_u.
\end{equation}
for all $s,t\in[0,T]$.
\end{defn}

\begin{rem}
  We stress that (3.2) is an equation in $\D^{N_\gamma\gamma}_{\bW}$, which in fact implies that $\langle e_w^*,\bX_t\rangle=F_w(X_t)$ for all words $w$ with $|w|\le N_\gamma-1$.
\end{rem}

\begin{rem}
  If $\bX\in\D^{N_\gamma\gamma}_{\bW}$ satisfies \cref{equ:forw_fixed_point}, it can also be regarded as an element of $\D^{(N_\gamma+1)\gamma}_{\bW}$, by \cref{eqn:rintdef}.
  Therefore we freely treat solutions to RDEs as elements of either of these spaces.
\end{rem}

By solving a fixed point equation on $\left( \D_{\bW}^{N_\gamma\gamma} \right)^n$ (see e.g. \cite{FH2014}) of the form
\[\bX_t= \bX_0+ \sum_{i=1}^d\int_0^t\bF_i(\bX)_u\,\mathrm d\bW^i_u  \]
with (see below for the definition of the functions $F_w\colon\R^n\to\R^n$)
\[ \bX_0=\sum_{|w|\le N_\gamma-1}F_w(x)e_w\in\left( H_{N_\gamma-1} \right)^n. \]
We can prove that there exists a unique global solution of \eqref{equ:forw_fixed_point} if the vector fields are of class $\CC^{N_\gamma+1}_b$. We recall this interesting expansion of the solution.
\begin{prop}[Davie's expansion]\label{prop:Davie}
A path $X\colon[0,T]\to \R^n$ is the unique rough path solution to \cref{eqn:RDE} in the sense of \Cref{defn:rdesol} if and only if
\begin{equation}\label{eqn:rdesol_local}
X_t\underset{(N_{\gamma}+1)\gamma}{=}\sum_{0\le|w|\le N_{\gamma}}F_w(X_s)\langle\bW_{st}, e_w\rangle
\end{equation}
and the coefficients of its lift $\bX\in(\D_{\bW}^{N_\gamma+1})^n$ are given by $\langle e_w^*,\bX_t\rangle=F_w(X_t)$ where the functions $F_w\colon\R^n\to\R^n$ are recursively defined by by $F_{\varepsilon}\coloneq{\id}$ and
\begin{equation}\label{Gamma_identity}
F_{iw}(x)\coloneqq DF_w(x)f_i(x).
\end{equation}

\end{prop}
\begin{rem}
By \cref{eqn:ctrlcoord} this results actually implies the chain of estimates, for all words $|w|\le N_{\gamma}$,
\[ F_w(X_t)\underset{(N_{\gamma}+1-|w|)\gamma}{=}\sum_{0\leq |u|\le N-|w|}F_{wu}(X_s)\langle\bW_{st},e_u\rangle. \]
\end{rem}
\begin{proof}[Proof of \Cref{prop:Davie}]
  Suppose that $\bX$ is a rough solution to \cref{eqn:RDE} in the sense of \Cref{defn:rdesol}. We define the functions $F_w\colon\R^n\to\R^n$ recursively by $F_i(x)\coloneq f_i(x)$ and
\begin{equation}
 F_{wi}(x)\coloneq\sum_{k=1}^{|w|}\frac{1}{k!}\smashoperator[r]{\sum_{\substack{u_1,\dotsc,u_k\\w\in\Sh(u_1,\dotsc,u_k)}}}D^kf_i(x)(F_{u_1}(x),\dotsc,F_{u_k}(x))
  \label{eqn:vectorfields}
\end{equation}
 Now it is an easy but tedious verification to show that these functions satisfy $F_{iw}(x)=DF_w(x)f_i(x)$; this identity essentially amounts to a reiterated use of the Leibniz rule.
 The form of the coefficients of $\bX$ is shown by induction, it being clear for a single letter $i=1,\dotsc,d$.
 If $w$ is any word with $0\le|w|\le N-1$ and $i\in\{1,\dotsc,d\}$ by definition
  \begin{align*}
    \langle e_{wi}^*,\bX_t-\bX_0\rangle&= \left\langle e_{wi}^*,\sum_{j=1}^d\int_0^t\bF_j(\bX)_u\,\mathrm dW^j_u \right\rangle\\
    &= \langle e_w^*,\bF_i(\bX)_t\rangle
  \end{align*}
where, in the second identity we have used \cref{eqn:roughint_coef}.
By \Cref{thm:ito}, the last coefficient equals
\[ \sum_{k=1}^{|w|}\frac{1}{k!}\sum_{\substack{u_1,\dotsc,u_k\\w\in\Sh(u_1,\dotsc,u_k)}}D^kf_i(X_t)(t;u_1,\dotsc,u_k)=F_{wi}(X_t) \]
by the induction hypothesis.
Then we obtain \cref{eqn:rdesol_local} from \Cref{defn:cRP} and \Cref{rem:coordinates}.

  Conversely, suppose that $X$ admits the local expansion in \cref{eqn:rdesol_local} and that the path $\bX$ satisfies $\langle e_w^*,\bX_t\rangle=F_w(X_t)$ for all words $w$ with $|w|\le N$.
  First we show that $X$ is controlled by $\bW$ with coefficients given by $\bX$.
  For this we have to Taylor expand the difference $F_w(X_t)-F_w(X_s)$ and collect terms as in the proof of \Cref{thm:ito}.
  Then, by \cref{eqn:roughint_coef} it is not difficult to see that in fact
  \[ \langle e_{wi}^*,\bX_t\rangle=F_{wi}(X_t)=\left\langle e_{wi}^*,\sum_{j=1}^d\int_0^t\bF_j(\bX)_u\,\mathrm dW^j_u \right\rangle \]
  so that \Cref{defn:rdesol} is satisfied.
\end{proof}

\subsection{Differentiability of the flow}
It is a standard result in classical ODE theory that given a regular enough vector field $V$, the equation $\dot X=V(X)$ induces a smooth flow on $\R^d$.
Indeed, if we let $X^x_t$ denote the unique solution of this equation such that $X_0^x=x$, then the map $(t,x)\mapsto X_t^x$ is a flow, in the sense that $(t,X^x_s)\mapsto X^x_{t+s}$ and the mapping $x\mapsto X_t^x$ is a diffeomorphism for each fixed $t$.
More precesily, if $V$ is of class $\CC^k$, then the application $x\mapsto X_t^x$ is also of class $\CC^k$.

Now we show that a similar statement is true in the case of RDEs.
The statement is the following
\begin{thm}
  Let $f_1,\dotsc,f_d$ be a family of class $\CC^{N_\gamma+1+k}_b$ vector fields in $\R^d$ for some integer $k\ge 0$, and $\bW\in\Cr^\gamma$.
  Then
  \begin{enumerate}
    \item the RDE
      \[ \mathrm dX_t=\sum_{i=1}^df_i(X_t)\,\mathrm d\bW^i_t,\quad X_s=x \]
      has a unique solution $\bX^{s,x}\in\D^{(N_\gamma+1)\gamma}_{\bW}$,
    \item the induced flow $x\mapsto X^{s,x}_t$ is a class $\CC^{k+1}$ diffeomorphism for each fixed $s<t$, and
    \item the partial derivatives satisfy the system of RDEs
      \begin{equation}
        \mathrm d\partial^\alpha X_t^{s,x}=\sum_{i=1}^d\sum_{k=1}^{|\alpha|}\frac{1}{k!}\sum_{\alpha\in \operatorname{Sh}(\beta_1,\dotsc,\beta_k)}D^kf_i(X^{s,x}_t)(\partial^{\beta_1}X^{s,x}_t,\dotsc,\partial^{\beta_k}X^{s,x}_t)\,\mathrm d\bW^i_t
        \label{eqn:pdrde}
      \end{equation}
      with initial conditions $X^{s,x}_s=x$, $\partial^i X^{s,x}_s=e_i$ and $\partial^\alpha X^{s,x}_s=0$ for all words with $|\alpha|\ge 2$.
  \end{enumerate}
  \label{thm:smoothflow}
\end{thm}

\begin{proof}
Point 1. and 2. are standard results in rough paths as found e.g. in Chapter 11 in \cite{FrizVictoir}.
For the algebraic identity in 3., it suffices to show the results in the case $W$ is smooth.  Indeed, by standard arguments $\bW\in\Cr^\gamma$ can be approximated uniformly with uniform $\gamma$-H\"older rough path bound, and hence in $\Cr^{\gamma-\eta}$ for any $\eta>0$, while on the other hand the particular structure (cf. Chapter 11 in \cite{FrizVictoir}) of the system of (rough) differential equations guarantees uniqueness and global existence so that the limiting argument is justified.  
  
  It remains to show point 3. for $\bW$ smooth. We note that the integral representation of the solution
  \[ X^{s,x}_t=x+\sum_{i=1}^d\int_s^tf_i(X^{s,x}_u)\dot W^i_u\,\mathrm du \]
  holds.
  By \Cref{lem:faadibruno}, for any $\alpha=(\alpha_1,\dotsc,\alpha_m)$ and $s<u<t$, we have
  \begin{equation}
    \partial^\alpha X_{ut}^{s,x}=\int_u^t\sum_{i=1}^d\sum_{k=1}^{m}\frac{1}{k!}\sum_{\beta_1,\dotsc,\beta_k}D^kf_i(X^{s,x}_r)(\partial^{\beta_1}X^{s,x}_r,\dotsc,\partial^{\beta_k}X^{s,x}_r)\dot W^i_r\,\mathrm dr
    \label{eqn:pdrde.1}
  \end{equation}
  which is the smooth version of \cref{eqn:pdrde}.
\end{proof}
We aim now to obtain a Davie-type expansion of the partial derivatives $\partial^\alpha X^{s,x}$ by making use of point 3. above.
  We observe that the above system of equations has the form
  \begin{align*}
    \mathrm dX^{s,x}_t&= \sum_{i=1}^df_i(X^{s,x}_t)\,\mathrm d\bW^i_t\\
    \mathrm dDX^{s,x}_t&= \sum_{i=1}^dDf_i(X^{s,x}_t)DX^{s,x}_t\,\mathrm d\bW^i_t\\
    \mathrm dD^2X^{s,x}_t&= \sum_{i=1}^dDf_i(X^{s,x}_t)D^2X^{s,x}_t\,\mathrm d\bW^i_t+(\cdots)\\
    &\vdots
  \end{align*}
  with initial conditions $X^{s,x}_s=x$, $DX^{s,x}_s=I$, $D^2X^{s,x}_s=D^3X^{s,x}_s=\dotsb=0$, where the inhomogeneity $(\cdots)$ is not important to spell out.
  
  The expansion is clear only for the first equation; it is just \cref{eqn:rdesol_local}.
  We would like to use \Cref{prop:Davie} to obtain an expansion of the second equation but the problem is that the vector field driving the equation depends on time, so the result does not directly apply.
  For the third and subsequent equations the problem is not only that but also they are non-homogeneous.
  
  To solve this problem we extend our state space $\R^n$ to (the still finite-dimensional space)
  \[ \mathfrak S_k\coloneq \R^n\oplus\mathcal L(\R^n,\R^n)\oplus\dotsb\oplus\mathcal L\left( (\R^n)^{\otimes(k-1)},\R^n \right) \]
  and define the vector fields (we give a more precise definition below in \cref{eqn:fidef}) $\mathfrak f_i\colon\mathfrak S_k\to\mathfrak S_k$ by
  \[ \mathfrak f_i(\mathfrak x)\coloneq (f_i(x),Df_i(x)(y_1),D^2f_i(x)(y_1,y_1)+Df_i(x)(y_2),\dotsc) \]
  where $\mathfrak x=(x,y_1,y_2,\dotsc,y_{k-1})\in\mathfrak S_k$.
  The previous proposition shows that if  
  \[\mathfrak X^{s,x}_t\coloneq (X^{s,x}_t,DX^{s,x}_t,\dotsc,D^{k-1}X^{s,x}_t) \]
  then
  \[ \mathrm d\mathfrak X^{s,x}_t=\sum_{i=1}^d\mathfrak f_i(\mathfrak X^{s,x}_t)\,\mathrm d\bW^i_t,\quad\mathfrak X^{s,x}_s\coloneq\mathfrak x=(x,I,0,\dotsc,0). \]
  This transformation turns the system of non-autonomous non-homogeneous RDEs into a single autonomous homogeneous RDE in $\mathfrak S_k$.

\begin{cor}
  \label{cor:partialdavie}
  For any word $\alpha$, the partial derivatives of the solution flow $X^{s,x}$ have the following Davie expansion: for any $p=1,\dotsc,k-1$, 
  \[ D^pX^{s,x}_{t}\underset{(N_\gamma+1)\gamma}{=}\sum_{0\le|v|\le N_\gamma}D^pF_w(x)\langle\bW_{st},e_w\rangle. \]
  In particular, for a word $\alpha\in\{1,\dotsc,n\}^p$ we have that
  \[ \partial^\alpha X^{s,x}_{t}\underset{(N_\gamma+1)\gamma}{=}\sum_{0\le|v|\le N_\gamma}\partial^\alpha F_w(x)\langle \bW_{st},e_w\rangle. \]
\end{cor}
\begin{proof}

  The hypotheses on the vector fields $f_1,\dotsc,f_d$ imply that $\mathfrak f_1,\dotsc,\mathfrak f_d$ are of class $\CC^{N_\gamma+1}_b$ on $\mathfrak S_k$, so this equation has a unique solution.
  Applying \Cref{prop:Davie} in this extended space we obtain, for $s<t$, the expansion
  \[ \mathfrak X_t^{s,x}\underset{(N+1)\gamma}{=}\sum_{0\le|w|<N}\mathfrak F_w(\mathfrak x)\langle\bW_{st},e_w\rangle. \]

 In order to deduce the result, we need to show that $\mathfrak F_w(\mathfrak x)_p=D^pF_w(x)$ for all words $w$ and $p=0,1,\dotsc,k-1$.
 We do this by induction on the length of $w$.
 If $w=i$ is a single letter, the $p$-th component, $p=0,1,\dotsc,k-1$, of the vector field $\mathfrak f_i$ is given by $\mathfrak f_i(x)_0=f_i(x)$ and
  \begin{equation}
\label{eqn:fidef}
    \mathfrak f_i(\mathfrak x)_p=\sum_{j=1}^p\sum_{(r)_j}\frac{p!}{r_1!\dotsm r_j!(1!)^{r_1}\dotsm (j!)^{r_k}}D^{p-j+1}f_i(x)(y_1^{r_1},\dotsc,y_{j}^{r_j})
  \end{equation}
  where the inner sum is over the set of indices $(r_1,\dotsc,r_j)$ such that $r_1+\dotsb+r_j=p-j+1$ and $r_1+2r_2+\dotsb+jr_j=p$.
  For our particular initial condition, $y_j=0$ for $j=2,3,\dotsc,k-1$ the formula simplifies to
  \[ \mathfrak f_i(\mathfrak x)_p=D^pf_i(x)\in\mathcal L\left( (\R^n)^{\otimes p},\R^n \right) \]
  since the only term left in \eqref{eqn:fidef} is the one with $j=1$, $r_1=p$.

  We continue by induction on the length of the word.
  We compute the $p$-th derivative of $x\mapsto F_{iw}(x)=DF_w(x)f_i(x)$ by recognizing that $F_{iw}=\varphi_1\circ\varphi_2$ with $\varphi_1(x,h)=DF_w(x)h$ and $\varphi_2(x)=(x,f_i(x))$.
  A quick check gives that the higher order derivatives of $\varphi_1$ and $\varphi_2$ are given by
  \begin{align*}
    D^m\varphi_1(x,h)( (u_1,v_1),\dotsc,(u_m,v_m) )&= D^{m+1}F_w(x)(u_1,\dotsc,u_m,h)+\sum_{j=1}^mD^mF_w(x)(u_1,\dotsc,\hat u_j,\dotsc,u_m)\\
    D^m\varphi_2(x)(h_1,\dotsc,h_m)&= (h_1\delta_{m=1},D^mf_i(x)(h_1,\dotsc,h_m))
  \end{align*}
  where $\hat u_j=v_j$.
  Thus, using \Cref{lem:faadibruno} we get that
  \[ D^pF_{iw}(x)(h_1,\dotsc,h_p)=\sum_{\pi\in\mathcal P(p)}D^{\#\pi}\varphi_1(\varphi_2(x))(D^{|B_1|}\varphi_2(x)h_{B_1},\dotsc,D^{|B_q|}\varphi_2(x)h_{B_q}). \]

  Now we have three cases, depending on the number of blocks of the partition in the above summation:
  \begin{enumerate}
    \item $q=p$: there is a single partition with $p$ blocks, and each block is a singleton. In this case the term equals
      \[ D^{p+1}F_w(x)(h_1,\dotsc,h_p,f_i(x))+\sum_{j=1}^pD^pF_w(x)(h_1,\dotsc,Df_i(x)h_j,\dotsc,h_p). \]
    \item $q=1$: there is a single partition with one block, namely $\pi=\{1,\dotsc,p\}$. In this case the term equals
      \[ DF_w(x)[D^pf_j(x)(h_1,\dotsc,h_p)]. \]
    \item $1<q<p$: there is at least one block of size greater than one, which means that the first term in the expression for $D^m\varphi_1$ vanishes since at least one of $u_1,\dotsc,u_m$ vanishes. For the rest of the terms, the exact result depends on whether there is a block with exactly one block or not: if all blocks have more than one block then the whole expression vanishes; otherwise, we obtain one term for each of the blocks having size exactly one, and it is of the form
      \[ D^{\#\pi}F_w(x)(h_{B},D^{\#p-|B|}f_i(x)h_{\pi\setminus B}). \]
  \end{enumerate}
  In either case, using the induction hypothesis it is possible to show that each of the terms appearing are of the form $\partial_r\mathfrak F_w(\mathfrak x)_p\mathfrak f_i(\mathfrak x)_r$, which then means that $D^pF_{iw}(x)=[D\mathfrak F_w(\mathfrak x)\mathfrak f_i(\mathfrak x)]_p$ as desired.
  For example, the term
  \[ D^{p+1}F_w(h_1,\dotsc,h_p,f_i(x)) \]
  corresponds to
  \[  [\partial_0\mathfrak F_w(\mathfrak x)_p\mathfrak f_i(\mathfrak x)_0](h_1,\dotsc,h_p) \]
  and so on.

\end{proof}
  In particular for the first derivative, the first few terms of the expansion read
  \begin{align*}
    DX^{s,x}_t&= I+\sum_{i=1}^dDf_i(x)\langle\bW_{st},e_i\rangle+\sum_{i,j=1}^d\Bigl(Df_j(x)Df_i(x)+D^2f_j(x)(f_i(x),{\id})\Bigr)\langle\bW_{st},e_{ij}\rangle+\dotsb
  \end{align*}

\subsection{Itô's formula for RDE}\label{sse:ito_formula}
The last ingredient to add in the study of the rough transport equation is to write down a change of variable formula for a solution of \cref{eqn:RDE} for some sufficiently smooth vector field $f=(f_1,\dotsc,f_d)$. By analogy with terminology of stochastic calculus we call it an ``It\^o formula''.
For any $i=1,\dotsc,n$ we denote by  $\Gamma_i$ the differential operator $f_i(x)\cdot D_x$ and for any non-empty word $w=i_1\cdots i_m$  we use the shorthand notation
\[\Gamma_{w}:=\Gamma_{i_1}\circ \cdots \circ \Gamma_{i_m}\,.\]
Moreover we adopt the convention $\Gamma_{\varepsilon}={\operatorname{id}}$.
\begin{lem}\label{lem:comb2}
  Let $f_1,\dotsc,f_d\in\CC^{N_\gamma+1}(\R^n; \R^n)$ be vector fields on $\R^n$.
  If $\phi\colon\R^n\to\R$ is a smooth function and $w$ is a nonempty word, then
  \begin{equation}
    \Gamma_{w}\phi(x)=\sum_{k=1}^{\vert w\vert}\frac1{k!}\sum_{\substack{u_1,\dotsc,u_k\\w\in\operatorname{Sh}(u_1,\dotsc,u_k)}}D^k\phi(x)(F_{u_1}(x),\dotsc,F_{u_k}(x)).
    \label{eqn:Gammaphi}
  \end{equation}
\end{lem}

\begin{proof}
Before commencing we introduce some notation.
If $\phi\colon\R^n\to\R$ and $g_1,\dotsc,g_k\colon\R^n\to\R^n$ are smooth functions, we define
\[ D^k\phi(x):(g_1,\dotsc,g_k)\coloneq D^k\phi(x)(g_1(x),\dotsc,g_k(x)) \]
where the right-hand side was defined in \cref{eqn:diffk}.
The Leibniz rule then gives that for any $h\in\R^n$ we have
\[ h\cdot D_x\left(D^k\phi(x):(g_1,\dotsc,g_k) \right)=D^{k+1}\phi(x):(h,g_1,\dotsc,g_k)+\sum_{i=1}^kD^k\phi(x):(g_1,\dotsc,(D_xg_i)h,\dotsc,g_k). \]

We now prove the result by induction on the word's length $\vert w \vert$.
If $w=i$ is a single letter then $\Gamma_i\phi(x)=f_i(x)\cdot \nabla\phi(x)=D\phi(x)f_i(x)$ which is exactly \cref{eqn:Gammaphi}.
Supposing the identity true for any word $w'$  such that $\vert w'\vert \leq \vert w\vert$, we prove it for $jw$ where $j\in \{1,\dotsc,d\}$. By induction one has
\[\begin{split}
\Gamma_{w}\phi(x)&=\sum_{k=1}^{\vert w\vert}\frac1{k!}\sum_{\substack{u_1,\dotsc,u_k\\w\in\operatorname{Sh}(u_1,\dotsc,u_k)}}D^k\phi(x)(F_{u_1}(x),\dotsc,F_{u_k}(x)).
\end{split}\]
By the above form of Leibiz's rule, with $g_i=F_{u_i}$ and $h=f_j(x)$, and noticing that by definition
\[ D_xF_{u_i}(x)f_j(x)=F_{ju_i}(x) \]
we obtain that
\[\begin{split}
\Gamma_{j}\left( D^k\phi(x):(F_{u_1},\dotsc,F_{u_k})\right)&=D^{k+1}\phi(x):(f_j,F_{u_1},\dotsc,F_{u_k})\\&+\sum_{i=1}^kD^k\phi(x):(F_{u_1},\dotsc,F_{ju_i},\dotsc,F_{u_k})\,. 
\end{split}\]
Summing this expression over  words $u_1,\dotsc,u_k$, we can rewrite it as
\[\begin{split}
    &\sum_{r=1}^k \sum_{\substack{u_1,\dotsc,u_k\\ jw\in\operatorname{Sh}(u_1,\dotsc, ju_r, \dotsc,u_k)}}D^{k} \phi(x):(F_{u_1},\dotsc,F_{ju_r},\dotsc,F_{u_k})\\&+\frac{1}{k+1}\sum_{r=1}^{k+1}\sum_{\substack{u_1,\dotsc,u_k\\ jw\in\operatorname{Sh}(u_1,\dotsc, j, \dotsc,u_k)}}D^{k+1} \phi(x):(F_{u_1},\dotsm,\overbrace{f_{j}}^{\text{$r$th place}},\dotsc, F_{u_k}),
 \end{split}\]
 the factor $1/(k+1)$ is introduced because of the symmetry of $D^{k+1}\phi(x)$. Summing finally over $k$, we can express the final expression as
\[\begin{split}
    \Gamma_{jw}\phi(x)=&\sum_{k=1}^{\vert w\vert}\frac{1}{k!}\sum_{r=1}^k \sum_{\substack{u_1,\dotsc,u_k\\ jw\in\operatorname{Sh}(u_1,\dotsc, ju_r, \dotsc,u_k)}}D^{k} \phi(x):(F_{u_1},\dotsm,F_{ju_r},\dotsc, F_{u_k})\\&+\sum_{k=1}^{\vert w\vert} \frac{1}{(k+1)!}\sum_{r=1}^{k+1}\sum_{\substack{u_1,\dotsm,u_k\\ jw\in\operatorname{Sh}(u_1,\dotsc,j,\dotsc,u_k)}}D^{k+1}\phi(x):(F_{u_1},\dotsm,\overbrace{f_{j}}^{\text{$r$th place}},\dotsc,F_{u_k}).
 \end{split}\]
Since the letter $j$ may appear as a single word or concatenated at the right with some word, we finally identify the whole  expression above with
\[\begin{split}
&\sum_{k=1}^{\vert w\vert +1}\frac1{k!}\sum_{\substack{u_1,\dotsc,u_k\\ja\in\operatorname{Sh}(u_1,\dotsc,u_k)}}D^k\phi(x):(F_{u_1},\dotsc,F_{u_k}).
\end{split}\]
\end{proof}
Now we show a formula for the composition of the solution to the RDE \eqref{eqn:RDE} and a sufficiently smooth function.
\begin{thm}[Itô formula for RDEs]\label{rough_Ito}
  Let $f_i\in \CC^{N_\gamma+1}$ and let $\bX\in\D^{(N_\gamma+1)\gamma}_{\bW}$ be the unique solution of \cref{eqn:RDE} and $X_t=\langle\1^*,\bX_t\rangle$. Then for any real valued function $\phi\in \CC^{N_{\gamma}+1}_b(\R^n)$ one has the identity
\begin{equation}\label{ito}
\phi(X_t)=\phi(X_s)+ \sum_{i=1}^d\int_s^t(\Gamma_i\phi)(X_r)\,\mathrm d\bW^i_r\,.
\end{equation}
More generally, one has the following estimates at the level of controlled rough paths
\begin{equation}
\langle e^*_{w}, \mathbf{\Phi}(\bX)_t\rangle\underset{(N_{\gamma}+1-\vert w \vert) \gamma}{=}\langle e^*_{w},\mathbf{\Phi}(\bX)_s\rangle+ \left\langle e^*_w,\sum_{i=1}^d\int_s^t(\Gamma_i\mathbf{\Phi})(\bX)_r\,\mathrm d\bW^i_r\right\rangle,
\label{eqn:general_ito}
\end{equation}
where $\Gamma_i\mathbf{\Phi}(\bX)$ is the controlled lift of composition of $\bX$ with the function $\Gamma_i\phi\in\CC^{N_\gamma}$ and any non-empty word such that $\vert w\vert \leq N_{\gamma}$.
\end{thm}
\begin{proof}[Proof of Theorem \ref{rough_Ito}]
  The theorem is obtained by comparing the coefficients of the controlled rough paths $ \mathbf{\Phi}(\bX)_t$ and $\int_0^t(\Gamma_i\mathbf{\Phi})(X_r)\,\mathrm d\bW^i_r$ for every $i=1,\dotsc,d$. Using Lemma \ref{lem:comb2} and Proposition \ref{prop:Davie}, for every non-empty  word $a$ one has
\begin{equation}\label{proof_comp}
\begin{split}
  \langle e^*_{w},  \mathbf{\Phi}(\mathbf{X})_t\rangle&=\sum_{k=1}^{\vert w\vert}\frac1{k!}\sum_{\substack{u_1,\dotsc,u_k\\w\in\operatorname{Sh}(u_1,\dotsc,u_k)}}D^k\phi(X_t)(t;u_1,\dotsc,u_k)\\
  &=\sum_{k=1}^{\vert w\vert}\frac1{k!}\sum_{\substack{u_1,\dotsc,u_k\\w\in\operatorname{Sh}(u_1,\dotsc,u_k)}}D^k\phi(X_t):(F_{u_1},\dotsc,F_{u_k})\\
  &= \Gamma_w\phi(X_t).
\end{split}
\end{equation}
Using the same identities we also deduce for any word $w$,
\begin{equation}\label{proof_int}
\left\langle e^*_{wj},\sum_{i=1}^d\int_0^t(\Gamma_i\mathbf{\Phi})(\bX)_r\,\mathrm d\bW^i_r\right\rangle= \langle e^*_{w},(\Gamma_j\mathbf{\Phi})(\bX)_t\rangle= \Gamma_{w}(\Gamma_j\phi)(X_t)= \Gamma_{wj}\phi( X_t).
\end{equation}
Since $\sum_{i=1}^n\int_0^t(\Gamma_i\mathbf{\Phi})(\bX)_r\,\mathrm d\bW^i_r
$ and $\mathbf{\Phi}(\mathbf{X})_t$ belong both to $\D_{\bW}^{(N_{\gamma}+1)\gamma}$ for any word $w$ one has both
\begin{equation*}
  \langle e_w^*,\mathbf{\Phi}(\mathbf{X})_t\rangle- \langle e_w^*,\mathbf{\Phi}(\mathbf{X})_s\rangle \underset{(N_{\gamma}+1-|w|)\gamma}{=}\sum_{0< \vert v\vert\leq N_{\gamma}-|w|}\langle e_{wv}^*,\mathbf{\Phi}(\mathbf{X})_s\rangle\langle\bW_{st}, e_v\rangle,
\end{equation*}
and
\begin{equation*}
  \left\langle e_w^*,\sum_{i=1}^n\int_s^t(\Gamma_i\mathbf{\Phi})(\bX)_r\,\mathrm d\bW^i_r\right\rangle\underset{(N_{\gamma}+1-|w|)\gamma}{=}\sum_{0< \vert v\vert\leq N_{\gamma}-|w|}\left\langle e_{wv}^*,\sum_{i=1}^n\int_0^s(\Gamma_i\mathbf{\Phi})(\bX)_r\,\mathrm d\bW^i_r\right\rangle\langle\bW_{st}, e_v\rangle.
\end{equation*}
The identities \eqref{proof_comp} and \eqref{proof_int} imply that the right-hand sides of the above estimates are the same quantities.
Thus we obtain \cref{eqn:general_ito} by simply subtracting one side from the other. In case $w=\mathbf{1}$ one has
\[\phi(X_t)-\phi(X_s)- \sum_{i=1}^d\int_s^t(\Gamma_i\phi)(X_r)\,\mathrm d\bW^i_r\underset{(N_{\gamma}+1)\gamma}{=}0\,.\]
Since  $(N_{\gamma}+1)\gamma>1$ and the right hand side is the increment of a path, one has the identity \eqref{ito}.
\end{proof}
Using the identities \eqref{proof_comp} we can rewrite the It\^o formula using only the operators $\Gamma_w$.
\begin{cor}[It\^o-Davie formula for RDEs]
Let $X\colon[0,T]\to\R^n$ be the unique solution of  \cref{eqn:RDE}. Then for any real valued function $\phi\in \CC^{N_{\gamma}+1}_b(\R^n)$ and  any word $w$ one has the estimate
\begin{equation}\label{equ:discrete_Ito}
\Gamma_{w}\phi(X_t)   \underset{(N_{\gamma}+1-\vert w\vert ) \gamma}{=}\sum_{0\leq \vert v\vert\leq N_{\gamma}- \vert w\vert }\Gamma_{wv}\phi(X_s) \langle\mathbf{W}_{st},e_v\rangle.
\end{equation}
\end{cor}

\section{Rough transport and continuity}
\subsection{Rough transport equation}\label{sec:RTE} 
We now consider the rough transport equation
\begin{equation}\label{equ:Transport_RPDE}
\begin{cases}
-\mathrm du_s=\sum_{i=1}^d\Gamma_i u_s\,\mathrm d\bW^i_s,\\
u(T,\cdot)= g(\cdot)
\end{cases}
\end{equation}
where we recall the differential operator $\Gamma_i\coloneq f_i\cdot D_x$ for some vector fields $f_1,\dotsc,f_d$ on $\R^n$.

We now prepare the  definition of a regular solution to the rough transport equation.  Since we are in the fortunate position to have an explicit solution candidate we derive a graded set of rough path estimates that provide a natural generalisation of the classical  transport differential equation.

\begin{defn}\label{cor:RTEex}
  Let $\gamma\in (0,1)$, $\mathbf{W}\in \Cr^{\gamma}$ a weakly-geometric rough path of roughness $\gamma$ and $g\in\CC^{N_\gamma+1}$. A $\CC^{\gamma,N_\gamma+1}$-function $u\colon [0,T]\times \R^n\to\R$ such that $u(T, \cdot)= g(\cdot)$ is said to be a regular solution to the rough transport equation \eqref{equ:Transport_RPDE} if one has the estimates
\begin{equation}\label{equ:estimate_sol}
  \Gamma_{w} u_s (x)\underset{(N_{\gamma}+1-\vert w \vert) \gamma}{=}\sum_{0\leq \vert v\vert\leq N_{\gamma}-\vert w\vert}\Gamma_{ wv}u_t (x)  \langle\mathbf{W}_{st},e_v \rangle\,,
\end{equation}
for every $s<t\in [0,T]$, uniformly on compact sets in $x$ and any word $w$.
\end{defn}

\begin{rem}
Since each application of the vector fields $\Gamma_{i_1\cdots i_n}$ amounts to take $n$ derivatives, these estimates have the interpretation that time regularity of $\Gamma_{i_1\cdots i_n}u$, can be traded against space regularity in a controlled sense.
\end{rem}

\begin{thm} \label{prop:RTEex}
Let $f \in \CC^{2N_\gamma+1}_b$, $g\in \CC^{N_\gamma+1}$ and consider the rough solution $X^{s,x}$ to \cref{eqn:RDE}. Then $u(s, x)\coloneq g(X^{s,x}_T)$ is a solution to the rough transport equation in the sense of \Cref{cor:RTEex}.
\end{thm}
\begin{proof}
We first note that by \Cref{thm:smoothflow} the map $(s,x)\mapsto X^{s,x}_T$ belongs to $\CC^{\gamma,N_\gamma+1}$. Since $g\in\CC^{N_\gamma+1}$ then $u(s,x)= g(X^{s,x}_T)\in\CC^{\gamma,N_\gamma+1}$. Let us show that $u$ is a solution by proving the estimates given in \Cref{cor:RTEex} for some fixed times $s < t < T$ and $x$ in compact set. By uniqueness of the RDE flow one has $X^{s, x}_T = X^{t, y}_T $ where $y = X^{s,x}_t$. Thus we deduce from the definition of $u$  the identity
\begin{equation} \label{composition}
u_s(x) = u_t(X^{s,x}_t).
\end{equation}
Let $\bX$ denote the controlled rough path such that $X^{s,x}_t=\langle\1^*,\bX_t\rangle$. Since $g\in\CC_b^{N_{\gamma}+1}$, we can apply the rough It\^o formula in \cref{equ:discrete_Ito} to the function $x\to u_t(x)$ obtaining
 \[ u_t(X^{s,x}_t)\underset{(N_{\gamma}+1)\gamma}{=}\sum_{|w|\le N}\Gamma_w u_t(x)\langle\bW_{st},w\rangle \]
 obtaining \eqref{equ:estimate_sol} for the case of $w=\varepsilon$.
 To show the estimates on $\Gamma_{i_1\cdots i_l}u_s$, we apply Lemma \ref{lem:comb2} to the function $x\to u_s(x)$
\begin{equation}\label{higher_order}
\Gamma_{w}u_s(x) = \sum_{k=1}^{\vert w \vert  }\frac{1}{k!}\sum_{\substack{u_1,\dotsc,u_k\\ w\in\Sh(u_1,\dotsc,u_k)}}D^ku_s(x)(F_{u_1}(x),\dotsc,F_{u_k}(x)).
\end{equation}
Using again the identity \eqref{composition}, for any word $\alpha$ we apply \cref{Faa_di_bruno} obtaining
\[\partial^\alpha (u_s(x))= \sum_{l=1}^k \frac{1}{l!}\sum_{\substack{\beta_1,\dotsc,\beta_l\\ \alpha\in\operatorname{Sh}(\beta_1,\dotsc,\beta_l)}}D^lu_t(X^{s,x}_t)(\partial^{\beta_1}X^{s,x}_t,\cdots , \partial^{\beta_l}X^{s,x}_t).\]
Since the vector field $f\in C_b^{2N_{\gamma}+1}$ and every $\beta_i$ such that $ \alpha\in\operatorname{Sh}(\beta_1,\dotsc,\beta_l)$ satisfies $\vert \beta_i \vert \leq \vert a\vert$ we can apply \Cref{cor:partialdavie} to get
\[\partial^{\beta_i}X_t^{s,x}\underset{(N_{\gamma}+1-\vert w \vert)\gamma}{=}\sum_{0\leq |v|\leq  N_{\gamma}-\vert w \vert }\partial^{\beta_i}F_v(x)\langle\bW_{st}, e_v\rangle\,\]
Plugging these estimates in $D^lu_t(X^{s,x}_t)$ and one has
\begin{equation}
\begin{gathered}
 D^lu_t(X^{s,x}_t)(\partial^{\beta_1}X^{s,x}_t,\dotsc,\partial^{\beta_l}X^{s,x}_t)\underset{(N_{\gamma}+1-\vert w \vert)\gamma}{=}\\\sum_{\substack{0\leq \vert v_1\vert \cdots \vert v_l\vert \leq N_{\gamma}-\vert a\vert}}D^lu_t(X^{s,x}_t)\left(\partial^{\beta_1}F_{v_1}(x), \dotsc ,\partial^{\beta_l}F_{v_l}(x)\right)\langle\bW_{st},e_{v_1}\shuffle\dotsm\shuffle e_{v_l}\rangle.
\end{gathered}
\end{equation}
Plugging this expression into \eqref{higher_order} and we obtain
\begin{equation}\label{last_comb}
\begin{gathered}
\Gamma_w u_s(x)\underset{(N_{\gamma}+1-\vert w \vert)\gamma}{=}\sum_{k=1}^{\vert w \vert}\sum_{\substack{u_1,\dotsc,u_k\\ w\in\operatorname{Sh}(u_1,\dotsc,u_k)}} \sum_{\alpha_1\,, \cdots \,,\alpha_k=1}^n\sum_{l=1}^k \frac{1}{l!}\frac{1}{k!}F^{\alpha_1}_{u_1}(x)\dotsm F^{\alpha_k}_{u_k}(x)\\
 \sum_{\substack{0\leq \vert d_1\vert \cdots \vert d_l\vert \leq N_{\gamma}-\vert a\vert}}\sum_{\substack{\beta_1,\dotsc,\beta_l\\ \alpha\in\operatorname{Sh}(\beta_1,\dotsc,\beta_l)}}D^lu_t(X^{s,x}_t)(\partial^{\beta_1}F_{v_1}(x), \dotsc ,\partial^{\beta_l}F_{v_l}(x))\langle\bW_{st},e_{v_1}\shuffle\dotsm\shuffle e_{v_l}\rangle.
\end{gathered}
\end{equation}
Rearranging the sums and applying the definition of the functions $F_w$ we obtain the identity
\[\begin{split}
 &\sum_{k=l}^{\vert w \vert}\frac{1}{k!}\sum_{\substack{u_1,\dotsc,u_k\\ w\in\operatorname{Sh}(u_1,\dotsc,u_k)}} \sum_{\substack{\alpha\in\operatorname{Sh}(\beta_1,\dotsc,\beta_l)\\ \vert \alpha \vert=k}}D^lu_t(X^{s,x}_t)(\partial^{\beta_1}F_{v_1}(x), \dotsc ,\partial^{\beta_l}F_{v_l}(x))F^{\alpha_1}_{u_1}(x)\dotsm F^{\alpha_k}_{u_k}(x)\\&=\sum_{\substack{u'_1,\dotsc,u'_l\\ w\in\operatorname{Sh}(u'_1,\dotsc,u'_l)}}D^lu_t(X^{s,x}_t)(F_{u'_1v_1}(x), \dotsm ,F_{u'_lv_l}(x)).
 \end{split}\]
Therefore the right-hand side of \eqref{last_comb} becomes
\begin{equation}\label{almost_end}
\sum_{l=1}^{\vert w \vert} \sum_{\substack{0\leq \vert v_1\vert \cdots \vert v_l\vert \leq N_{\gamma}-\vert w\vert}}\sum_{\substack{u'_1,\dotsc,u'_l\\ w\in\operatorname{Sh}(u'_1,\dotsc,u'_l)}}\frac{1}{l!}D^lu_t(X^{s,x}_t)(F_{u'_1v_1}(x), \dotsm ,F_{u'_lv_l}(x))\langle\bW_{st},e_{v_1}\shuffle\dotsm\shuffle e_{v_l}\rangle.
\end{equation}
We perform now  a Taylor expansion of $D^lu_t(X^{s,x}_t)$ up to order $ N- \vert w \vert $ between $X_t^{s,x}$ and $x$, yielding for any words $u'_1,\cdots, u'_k$
\begin{equation}
\begin{split}
&D^lu_t(X_t^{s,x})\bigg(F_{u'_1v_1}(x), \dotsm ,F_{u'_lv_l}(x)\bigg)\underset{(N_{\gamma} +1-\vert w \vert )\gamma}{=}\\& \sum_{m=0}^{N-\vert w\vert}\frac{1}{m!}D^{l+m}u_t(x)\left((X_t^{s,x}-x)^{\otimes m} ,F_{u'_1v_1}(x), \dotsm ,F_{u'_lv_l}(x)\right).
\end{split}
\end{equation}
Plugging now the Davie expansion  \eqref{eqn:rdesol_local} truncated at order $N_{\gamma}-\vert w\vert$  into \eqref{almost_end} we have the following estimate
\begin{equation}\label{hard_combinatorics2}
\begin{split}&\Gamma_w (u_s(x))\underset{(N_{\gamma}+1 -\vert w \vert)\gamma}{=}\sum_{l=1}^{\vert w \vert}\sum_{m=0}^{N_{\gamma}-\vert w\vert}\frac{1}{l!}\frac{1}{m!}\sum_{\substack{u'_1,\dotsc,u'_l\\ w\in\operatorname{Sh}(u'_1,\dotsc,u'_l)}}\sum_{\substack{0\leq \vert v_1\vert \cdots \vert v_l\vert \leq N-\vert w\vert\\0<\vert z_1\vert \cdots \vert z_{m}\vert \leq N-\vert w \vert}}\\&D^{l+m}u_t(x):(F_{u'_1v_1}, \dotsc ,F_{z_1}, \dotsc)\langle\bW_{st},e_{v_1}\shuffle\dotsm\shuffle e_{z_1}\shuffle\cdots \rangle\,.
\end{split}
\end{equation}
Using the symmetry of $D^{l+m}u_t(x)$, we deduce
\[\frac{l!m!}{(l+m)!}\sum_{ I_l\sqcup J_m= \{1,\cdots,m+l\}} D^{m+l}\phi(x):(F_{u'_{i_1}v_{i_1}}, \dotsm ,  F_{z_{j_1}}, \cdots,F_{u'_{i_l}v_{i_{l}}}, \cdots).\]
Replacing this expression in the right-hand side of \eqref{hard_combinatorics2}, we can easily verify that the resulting expression is equal to the sum
 \[ \sum_{0\leq  \vert v \vert\leq N-\vert w\vert }\sum_{n=1}^{|w|+ \vert v\vert}\sum_{\substack{u_1,\dotsc,u_n \\ wv\in\operatorname{Sh}(u_1,\dotsc,u_n)}}\frac{1}{n!}D^n u_t(x):(F_{u_1}, \dotsm ,F_{u_n})\langle\bW_{st},e_{v}\rangle.\]
Thereby proving the result.
\end{proof}

We can now show that solutions in the sense of \Cref{cor:RTEex} are unique.

\begin{thm}
Let $f_i \in \CC^{2N_{\gamma}+1}_b$ with associated differential operators $\Gamma_i$, and $\bW \in \Cr^\gamma$. Given regular terminal data $g \in \CC^{N_{\gamma}+1}$, there exists a unique regular solution to the rough transport equation \eqref{equ:Transport_RPDE}.
\end{thm}

\begin{proof}
Existence is clear, since Proposition \ref{prop:RTEex} exactly says that $(t, x) \mapsto g (X^{t, x}_T)$ gives a regular solution.
Let now $u$ be any solution to the rough transport equation. We show that, whenever $X=X^{\bar{s}, \bar{y}}$  for every $\bar{s}, \bar{y} $ one has the estimate
\begin{equation}\label{equ:uniqueness}
u (t, X_t) - u (s, X_s) \underset{(N_{\gamma}+1) \gamma}{=} 0.
\end{equation}
Since $(N_{\gamma}+1) \gamma > 1$ this entails that $t \mapsto u (t, X_t)$ is  constant, and so we recover the uniqueness from the identities
\[u (s, x) =u (s, X^{s, x}_s) =u (T, X^{s, x}_T) = g (X^{s, x}_T)\,.\]
To prove \eqref{equ:uniqueness} we show that for every $k = 0,\cdots, N_{\gamma}$  and any choice of indexes $i_1,\cdots, i_k $ (if $k=0$ we do not consider indexes) one has  the estimates
\[ \Gamma_{i_1\cdots i_k} u_t (X_t) \underset{(N_{\gamma}+1-k) \gamma}{=} \Gamma_{i_1\cdots i_k}u_s (X_s).\]
Let us prove this estimate by reverse induction on the indices length. The case when the indices $i_1\cdots i_{N_{\gamma}}$ have length $N_{\gamma}$ comes easily from the algebraic manipulation
\[\begin{split}
\Gamma_{i_1\cdots i_{N_{\gamma}}} u_t (X_t) - \Gamma_{i_1\cdots i_{N_{\gamma}}}u_s (X_s) &=  \bigg(\Gamma_{i_1\cdots i_{N_{\gamma}}} u_t (X_t) - \Gamma_{i_1\cdots i_{N_{\gamma}}} u_s (X_t)\bigg) \\&+\bigg( \Gamma_{i_1\cdots i_{N_{\gamma}}} u_s (X_t) - \Gamma_{i_1\cdots i_{N_{\gamma}}} u_s (X_s) \bigg)\,.
\end{split}\]
Using the defining property of a solution in the estimates \eqref{equ:estimate_sol}, the first difference on the right-hand side is of order $\gamma$. Moreover by hypothesis on $u$ one has $\Gamma_{i_1\cdots i_{N_{\gamma}}}u_s (\cdot)\in C^1$, always uniformly in $s \in [0, T]$, therefore the second difference is also of order $\gamma$, as required. Supposing the estimate true for every indices of length $k$ we will prove it on every indices $i_1\cdots i_{k-1} $ of length $k-1$ . By  repeating the same procedure as before we obtain
\[\begin{split}
 \Gamma_{i_1\cdots i_{k-1}} u_t (X_t)-\Gamma_{i_1\cdots i_{k-1}}u_s (X_s) &=  \underbrace{\bigg(\Gamma_{i_1\cdots i_{k-1}} u_t (X_t) - \Gamma_{i_1\cdots i_{k-1}} u_s (X_t)\bigg)}_I \\&+\underbrace{\bigg( \Gamma_{i_1\cdots i_{k-1}} u_s (X_t) - \Gamma_{i_1\cdots i_{k-1}} u_s (X_s) \bigg)}_{II}.
 \end{split}\]
Using the definition of a solution, the first difference on the right-hand side satisfies
\[I\underset{({N_{\gamma}}+1-k) \gamma}{=}- \sum_{k=1}^{{N_{\gamma}}+1-k} \sum_{\vert w\vert=k }\Gamma_{i_1\cdots i_{k-1} w}u_t (X_t)  \langle\mathbf{W}_{st},w\rangle.\]
On the other hand, using Lemma \ref{lem:comb2} two times we write $\Gamma_{i_1\cdots i_{k-1}} u_s(X_t)=  \langle e^*_{i_1\dotsm i_{k-1}}, \mathbf{U}_s(\mathbf{X})_t\rangle$ so that the second difference can be replaced by the usual remainder
\[ \begin{split}
II \underset{({N_{\gamma}}+1-k) \gamma}{=}&\sum_{k=1}^{{N_{\gamma}}+1-k} \sum_{\vert w\vert=k }\langle e^*_{ i_1\cdots i_{k-1}w}, \mathbf{U}_s(\mathbf{X})_s\rangle\langle \mathbf{W}_{st},w  \rangle\\\underset{({N_{\gamma}}+1-k) \gamma}{=}&\sum_{k=1}^{{N_{\gamma}}+1-k} \sum_{\vert w\vert=k }\Gamma_{ i_1\cdots i_{k-1}w}u_s (X_s)  \langle\mathbf{W}_{st},w  \rangle.\end{split}\]
Combining the two estimates we obtain
\[ I+II= -\sum_{k=1}^{{N_{\gamma}}+1-k}\sum_{\vert w\vert=k }\bigg(\Gamma_{ i_1\cdots i_{k-1}w}u_t (X_t) -\Gamma_{ i_1\cdots i_{k-1}w}u_s (X_s) \bigg) \langle\mathbf{W}_{st},w  \rangle\,. \]
Since the terms in the sum involve the increment $ \Gamma_{\sigma} u_t (X_t)- \Gamma_{\sigma} u_s (X_s)$ where $\sigma$ has length bigger or equal than $k$ we apply the recursive hypothesis obtaining that each term satisfies
\[\Gamma_{i_1\cdots i_{k-1}w} u_t (X_t)- \Gamma_{ i_1\cdots i_{k-1}w} u_s (X_s)\underset{({N_{\gamma}}+1-k-\vert w\vert ) \gamma}{=} 0 \] and the multiplication with $\langle\mathbf{W}_{st},w  \rangle$ gives the desired estimate.
\end{proof} 
\subsection{Continuity equation and analytically weak formulation} \label{sec:RCTE}
\index{stochastic contuity equation} \index{rough!continuity equation} \index{continuity equation}
Given a finite measure $\rho \in \mathcal{M} (\R^n)$ and a continuous bounded function $\phi \in C_b (\R^n)$, we write $\rho (\phi) = \int \phi (x) \rho (dx)$ for the natural pairing. We are interested in measure-valued (forward) solutions to the continuity equation
\[\begin{cases}
    d_t \rho_t = \displaystyle\sum_{i = 1}^d \div_x (f_i (x) \rho_t)\,\mathrm d\bW_t^i   \quad\text{ in }\quad(0,T)\times\R^{n},\\[2mm]
\rho_0=\mu \quad\text{ on }\quad\{0\}\times\R^{n}
\end{cases}
\]
when $\bW$ is again a weakly geometric rough path.  As before we use the notation $\Gamma_i = f_i (x) \cdot D_x$, whose formal adjoint  is $\Gamma_i^\star = -\div_x (f_i{\cdot})$.

\begin{defn}
Let $\gamma\in (0,1)$, $\mathbf{W}\in \Cr_g^{\gamma}$ and $\mu\in  \mathcal{M} (\R^n)$. Any function $\rho\colon [0,T]\to \mathcal{M} (\R^n)$ such that $\rho_0= \mu$ is called a weak or measure-valued solution to the  rough continuity equation
\begin{equation}\label{RCE}
  \mathrm d\rho_t = \sum_{i = 1}^d \div_x (f_i (x) \rho_t)\,\mathrm d\bW^i_t
\end{equation}
if for every $\phi$ bounded in $\CC^{{N_{\gamma}}+1}_b$ and any word $w$ with $|w|\le N_\gamma$ one has  the estimates
\begin{equation}
\rho_t(\Gamma_w\phi)\underset{(N_\gamma+1-|w|)\gamma}{=}\sum_{0\le|v|<N_\gamma+1-|w|}\rho_s(\Gamma_{wv}\phi)\langle\bW_{st},e_v\rangle
  \label{eqn:weaktransport}
\end{equation}
for every $s<t\in [0,T]$ and uniformly in $\phi$.
\end{defn}

\begin{thm} \label{thm:RCE}
  Let $f \in \CC^{2{N_{\gamma}}+1}_b$ and $\bW \in \Cr_g^\gamma$. Given initial data $\mu \in \mathcal{M} (\R^n)$, there exists a unique solution to the measure-valued rough continuity equation, explicitly given for $\phi \in \CC_b^{{N_{\gamma}}+1}$ by
$$
\rho_t (\phi) = \int \phi (X^{0, x}_t) \mu (d x)\;,
$$
where $ X^{0, x}$ is the unique solution of the RDE $\mathrm dX_t = \sum_{i=1}^d f_i(X_t)\,\mathrm d\bW^i_t$ such that $X^{0,x}_0=x$.
\end{thm}

\begin{proof}
{\it (Existence)} Using the composition of the controlled rough path $\mathbf{X}^{0,x}$ with $\phi \in \CC^{{N_{\gamma}}+1}_b$ and the shorthand notation $X^{0,x}_t=X_t$ we can write
\[\phi(X_t)  \underset{({N}_{\gamma}+1) \gamma}{=} \phi(X_s)  + \sum_{k=1}^{{N_{\gamma}}} \sum_{\vert w\vert=k } \Gamma_w\phi(X_s)\langle\mathbf{W}_{st},w \rangle\,,\]
\[\Gamma_{i_1\cdots i_n}\phi(X_t)   \underset{(N_{\gamma}+1-n) \gamma}{=}\Gamma_{i_1\cdots i_n}\phi(X_s) +\sum_{k=1}^{{N_{\gamma}}-n} \sum_{\vert w\vert=k }\Gamma_{i_1\cdots i_n  w}\phi(X_s) \langle\mathbf{W}_{st},w  \rangle\,.\]
This showing the existence when $\mu = \delta_x$ thanks to Proposition \ref{prop:RTEex}. Since we are dealing with bounded vector fields, all these estimates are uniform in $X_0 = x$. Thus we can integrate both sides with respect to the measure  $\mu$, obtaining the existence.
\medskip

\noindent
{\it (Uniqueness)} To prove the uniqueness, we will show that for any $0<t\leq T$, any function $g \in \CC^{{N_{\gamma}}+1}_b$ and any solution $u\colon [0,t]\times \R^n\to\R$ of the RPDE
\[
\mathrm d u_r = \displaystyle\sum_{i=1}^d \Gamma_i u (r, x)\,\mathrm d\bW^i_r,\quad  u_t=g,
\]
the function $r\in [0,t] \mapsto \alpha(r)\coloneq\rho_r (u_r)$ is constant. This property implies that for any function $g \in \CC^{{N_{\gamma}}+1}_b$ and $t>0$ one has the identity
\[\rho_t (g)=\rho_t (u_t) = \rho_0 (u_0)= \mu(u_0)\]
which uniquely determines the measure $\rho_t$ for any $0<t\leq T$. Since the parameter $T$ was also arbitrary it is not restrictive to prove the result when $t=T$. Then $\alpha$ is constant if and only if one has the estimate
\begin{equation}\label{final_identity}
\alpha(r)\underset{(N_{\gamma}+1) \gamma}{=} \alpha(s)\,.
\end{equation}
Writing  $u_{s,r} = u_r - u_s$ and similarly for $\rho$ one has
\[ \rho_r (u_r) - \rho_s (u_s) = \rho_{s, r} (u_r)+\rho_s (u_{s, r}) \,.\]
By construction of regular solution with $\phi = u_r\in \CC^{{N_{\gamma}}+1}_b$ the first summand expands as
\begin{equation}\label{thm:uni_RCE2}
\rho_{s, r}( u_r) \underset{({N_{\gamma}}+1) \gamma}{=}  \sum_{k=1}^{N_{\gamma}} \sum_{\vert w\vert=k }\rho_s (\Gamma_w u_r)  \langle\mathbf{W}_{sr},w \rangle \,. \end{equation}
On the other hand, we expand the second summand on the right-hand using the very definition of regular backward RPDE obtaining
\[u_{s,r} (x)  \underset{(N_{\gamma}+1) \gamma}{=} -\sum_{k=1}^{{N_{\gamma}}} \sum_{\vert w\vert=k }\Gamma_{ w}u_r (x)  \langle\mathbf{W}_{sr},w \rangle \,,\]
where the remainder is uniform on $x$. By integrating this estimate on $\rho_s$, we obtain
\begin{equation}\label{thm:uni_RCE}
\rho_s (u_{s, r})\underset{(N_{\gamma}+1) \gamma}{=} -\sum_{k=1}^{{N_{\gamma}}} \sum_{\vert w\vert=k }\rho_s(\Gamma_{ w}u_r )  \langle\mathbf{W}_{sr},w \rangle \,.
\end{equation}
Combining the two estimates \eqref{thm:uni_RCE} and \eqref{thm:uni_RCE2} we obtain \eqref{final_identity} and the theorem is proven.
\end{proof}

\bibliographystyle{amsalpha}
\bibliography{refs}

\end{document}